\documentclass{amsart}
\usepackage{amsmath,amssymb,amscd,latexsym}

\usepackage[all]{xy}
\newcommand{\A}{{\mathbb{A}}}
\newcommand{\C}{{\mathbb{C}}}
\newcommand{\Lee}{\mathbb{L}}

\newcommand{\Qell}{{\mathbb{Q}}_{\ell}}

\newcommand{\Pro}{{\mathbb{P}}}

\newcommand{\Z}{{\mathbb{Z}}}
\newcommand{\Zl}{{\mathbb{Z}}/\ell}

\newcommand{\Zell}{{\mathbb{Z}}_{\ell}}

\newcommand{\car}{\mathrm{char}\;}
\newcommand{\clH}{\mathrm{cl}_{H}}
\newcommand{\clMU}{\mathrm{cl}_{MU}}

\newcommand{\et}{\mathrm{\acute{e}t}}

\newcommand{\Et}{\mathrm{Et}\,}
\newcommand{\hEt}{\hat{\mathrm{Et}}\,}

\newcommand{\compl}{\hat{(\cdot)}}

\newcommand{\Gal}{\mathrm{Gal}}

\newcommand{\Hom}{\mathrm{Hom}}

\newcommand{\sk}{\mathrm{sk}}

\newcommand{\Sm}{\mathrm{Sm}}

\newcommand{\Hh}{{\mathcal H}}

\newcommand{\hHh}{\hat{{\mathcal H}}}
\newcommand{\hHhp}{\hat{{\mathcal H}}_{\ast}}

\newcommand{\Oh}{{\mathcal O}}
\newcommand{\Ph}{\mathcal{P}}

\newcommand{\Sh}{{\mathcal S}}

\newcommand{\SHh}{{\mathcal{SH}}}

\newcommand{\hSHh}{{\hat{\mathcal{SH}}}}

\newcommand{\hSh}{\hat{\mathcal S}}

\newcommand{\hShp}{\hat{\mathcal S}_{\ast}}

\newcommand{\hSp}{\mathrm{Sp}(\hShp)}

\newcommand{\Sp}{\mathrm{Sp}(\Sh_{\ast})}

\newcommand{\ok}{\bar{k}}

\newcommand{\oW}{\Bar{W}}
\newcommand{\oX}{\Bar{X}}
\newcommand{\oY}{\Bar{Y}}

\newcommand{\ovZ}{\Bar{Z}}

\newcommand{\hMU}{\hat{M}U}

\newcommand{\hBP}{\hat{B}P}

\newcommand{\into}{\hookrightarrow}
\newtheorem{theorem}{Theorem}[section]
\newtheorem{lemma}[theorem]{Lemma}
\newtheorem{prop}[theorem]{Proposition}

\newtheorem{cor}[theorem]{Corollary}
\newtheorem{convention}[theorem]{Convention}

\newtheorem{remark}[theorem]{Remark}
\begin{document}
\title{Torsion algebraic cycles and \'etale cobordism}
\author{Gereon Quick}
\date{}
%\maketitle
\begin{abstract}
We prove that the classical integral cycle class map from algebraic cycles to \'etale cohomology factors through a quotient of $\ell$-adic \'etale cobordism over an algebraically closed field of positive characteristic. This shows that there is a strong topological obstruction for cohomology classes to be algebraic and that examples of Atiyah, Hirzebruch and Totaro also work in positive characteristic. 
\end{abstract}

\maketitle

\section{Introduction}

Atiyah and Hirzebruch \cite{atiyahhirzebruch} showed that an integral cohomology class of a complex variety $X$ has to satisfy certain conditions in order to be algebraic. If a cohomology class $y$ in $H^{\ast}(X;\Z)$ is algebraic, all differentials $d_ry$ in the spectral sequence $H^{\ast}(X;\Z)\Rightarrow K^{\ast}(X)$ to topological $K$-theory vanish, or, in other terms, all primary odd degree cohomology operations vanish on $y$. Moreover, they showed that these conditions are not vacuous by constructing examples using Godeaux-Serre varieties. Therefore, they showed that the integral version of the Hodge conjecture for complex varieties fails in general. Recently, Totaro \cite{totaro} revisited the obstructions of Atiyah and Hirzebruch and showed that they are induced by a stronger condition. Totaro proved that the classical map from cycles on $X$ to integral cohomology factors through some quotient of complex cobordism as
$$CH^{\ast}X\to MU^{2\ast}X\otimes_{MU^{\ast}}\Z \stackrel{\theta}{\to} H^{2\ast}(X;\Z).$$
Hence for an integral even degree cohomology class to be algebraic, it has to be in the image of the canonical map $\theta:MU^{2\ast}X\otimes_{MU^{\ast}}\Z \to H^{2\ast}(X;\Z)$. Since also all higher order odd degree cohomology operations vanish on the image of $\theta$, this obstruction is stronger than the one of \cite{atiyahhirzebruch}. Moreover, Totaro showed that it provides a method to construct nontrivial cycles in the Griffiths group of certain varieties. Voisin discusses this topological obstruction and other constructions of counter-examples for the integral Hodge conjecture in \cite{voisin}.\\
Over a finite field, the analogue of the Hodge conjecture is the Tate conjecture. Let $k$ be a finite field, $\ok$ its algebraic closure and $G:=\Gal(\ok/k)$ the absolute Galois group of $k$. Let $\ell$ be a prime different from the characteristic of $k$. For a projective smooth and geometrically integral variety $X$ over $k$, let us consider the integral version of the Tate conjecture and ask if the homomorphism
\begin{equation}\label{tateintegral}
CH^iX \otimes_{\Z} \Zell \to H_{\et}^{2i}(X_{\ok};\Zell(i))^G
\end{equation}
is surjective. As for the integral Hodge conjecture, this map (\ref{tateintegral}) is in general not surjective. For any algebraically closed field $k$ and a projective smooth variety $X$ over $k$, Colliot-Th\'el\`ene, Szamuely and Totaro \cite{ctsz} have shown that, for any prime $\ell \neq \car k$, all primary odd degree Steenrod operations vanish on algebraic cohomology classes in $H_{\et}^{2i}(X;\Z/\ell(i))$. Since Godeaux-Serre varieties are defined over any infinite field, this shows moreover that the initial examples of Atiyah and Hirzebruch \cite{atiyahhirzebruch} yield examples over any algebraically closed field for which the map 
$$CH^iX \otimes_{\Z} \Zell \to H_{\et}^{2i}(X;\Zell(i))$$ 
is not surjective, cf. \cite{ctsz}, Theorem 2.1.\\
The main result of this paper is that in fact Totaro's stronger obstruction has an analogue over an algebraically closed of arbitrary characteristic. Once the necessary theory is developed, the proof is similar to the one of Totaro's theorem in \cite{totaro}. But there are two main new highly nontrivial ingredients. First, we have to replace the usage of the analytic topology on complex varieties. Its natural analogue over fields of positive characteristic is the \'etale topology. By Artin, Mazur \cite{artinmazur} and Friedlander \cite{fried}, the information of the \'etale topology can be collected in a functor that associates to any locally noetherian scheme an \'etale homotopy type. For a fixed prime number $\ell$ different from the characteristic of the base field, we will use an $\ell$-adically completed version of this functor. Substituting the simplicial spectrum representing complex cobordism by an $\ell$-adic completion $\hMU$, we get the $\ell$-adically completed cobordism $\hMU_{\et}^{\ast}(X)$ of the \'etale homotopy type of $X$. The key point is then to construct a fundamental \'etale bordism class $[X]\in \hMU^{\et}_{2n}(X)$ for any smooth projective variety $X$ of dimension $n$. This involves the techniques developed in \cite{etalecob} and a new Poincar\'e duality theorem for \'etale bordism. To use this $\ell$-adically completed \'etale topological type as in \cite{etalecob} and not the usual pro-simplicial set is a significant progress in applying \'etale topological cohomology theories, as it yields simplified constructions and makes it possible to apply results from $\A^1$-homotopy theory, see \cite{etalecob} and the third section below. Furthermore, the proof of Theorem \ref{maintheoremintro} below, will demonstrate and use the full power of the analogy to the complex topological theory provided by the \'etale topological techniques of \cite{fried} and \cite{etalecob}, further developed in the third section.\\
Secondly, for a complex variety X, the class of a prime cycle, i.e. an irreducible subvariety $Z\subset X$, is defined to be the cobordism class of a resolution of $Z$. Over a field of positive characteristic, resolutions of singularities are not available. Therefore, we replace resolutions by smooth alterations. An alteration of a variety $X$ over a field $k$ is a proper dominant morphism $X' \to X$ of varieties over $k$ with $\dim X' =\dim X$. This is a weaker notion than a resolution since finite extensions of the function field are allowed. The existence of alterations such that $X'$ is smooth over the base field is part of the famous theorem of de Jong \cite{dejong}. But this result alone would not help us, since it does not provide any control of the degree of the alteration. In his recent studies on finiteness in \'etale cohomology and uniformizations \cite{gabber} \cite{illusie}, Gabber also proved that, for any prime $\ell$ different from the characteristic of the base field, there exists a smooth alteration of degree prime to $\ell$. Since $\hMU_{\et}^{\ast}(X)$ is a module over the coefficient ring $\hMU^{\ast}=MU^{\ast}\otimes_{\Z}\Zell$, the degree becomes invertible in $\ell$-adic cobordism. After constructing a fundamental \'etale cobordism class for any smooth projective variety over $k$, we get a well defined map by sending an irreducible subvariety $Z\subset X$ to the image of the fundamental class of an alteration $\pi:Z'\to Z$ in $\hMU^{\ast}(X)$ divided by the degree of $\pi$. The use of Gabber's theorem is exactly what is needed to get an $\ell$-adic integral version. Usually, replacing resolutions by alterations forces to switch to rational coefficients, because of the occuring nontrivial degrees of the maps. But as we explain below, the discovered topological obstruction is a torsion phenomenon and would vanish for rational coefficients. Hence to use Gabber's and not only de Jong's theorem is a key new idea for the proof.\\
This yields the following factorization of the cycle map in arbitrary characteristic.  
\begin{theorem}\label{maintheoremintro}
Let $k$ be an algebraically closed field and let $X$ be a smooth projective variety over $k$. Let $\ell$ be a prime different from the characteristic of $k$. There is a natural map $\clMU:Z^iX \to \hMU_{\et}^{2i}(X)\otimes_{MU^{\ast}} \Zell$ from codimension $i$ cycles on $X$ that vanishes on cycles algebraically equivalent to zero such that the composition
$$CH^{\ast} X \stackrel{\clMU}{\longrightarrow} \hMU_{\et}^{2\ast}(X)\otimes_{\hMU^{\ast}} \Zell \stackrel{\theta}{\longrightarrow} H_{\et}^{2\ast}(X;\Zell)$$
is the classical cycle map $\clH$ to continuos $\ell$-adic cohomology after choosing an isomorphism $\Zell\cong \Zell(1)$. This map $\clMU$ is compatible with pushforward maps for projective morphisms and commutes with intersection products.
\end{theorem}
This implies the following collection of results.
\begin{cor}\label{obstruction}
Let $k$ be an algebraically closed field and $\ell$ any prime different from the characteristic of $k$. We choose an isomorphism $\Zell\cong \Zell(1)$.\\
{\rm (a)} For a cohomology class of even degree to be algebraic, it has to be in the image of the map 
$$\hMU_{\et}^{2\ast}(X)\otimes_{\hMU^{\ast}} \Zell \stackrel{\theta}{\longrightarrow} H_{\et}^{2\ast}(X;\Zell).$$
In particular, all odd degree higher order cohomology operations on an algebraic cohomology class vanish.\\
{\rm (b)} There is a finite $\ell$-group $G$, a $G$-representation $V$, a smooth complete intersection $Y\subset \Pro(V)$ on which $G$ acts freely and a cohomology class $y$ of $\ell$-torsion in $H_{\et}^{4}(Y/G;\Zell)$ that is not algebraic.\\
{\rm (c)} If $\car k \neq 2$, there is a smooth projective varieties $X$ over $k$ of dimension $7$ such that the map $CH^2X/2 \to H_{\et}^4(X;\Z/2)$ is not injective. Moreover, there is a smooth projective variety $Y$ of dimension $15$ and a $2$-torsion cycle in $CH^3Y$ that is homologically but not algebraically equivalent to zero.
\end{cor}  
Note for part (a), that it is possible to construct higher cohomology operations in \'etale cohomology using the profinite \'etale homotopy type of $X$. The usual methods for spaces can be tranferred to profinite spaces.So far, it was only known that primary cohomology operations of odd degree vanish on the image of the cycle map in positive characteristic. Hence (a) yields a much stronger condition for algebraic cohomology classes over fields of positive characteristic.\\ 
Furthermore, all of the above examples of varieties in (b) and (c) are Godeaux-Serre varieties, constructed in \cite{serre}, for suitable finite groups. For (b), the group is just $G=(\Zl)^3$, so the varieties are defined in exactly the same way as by Atiyah and Hirzebruch in \cite{atiyahhirzebruch} over $k$ instead of $\C$. Note again that Colliot-Th\'el\`ene, Szamuely and Totaro have shown this already. But we show that the original proof of Atiyah and Hirzebruch in \cite{atiyahhirzebruch} fits very nicely in the picture of \'etale homotopy theory. Thus our proof might be closer to the remark of Milne \cite{milnetate}, Aside 1.4, that the arguments of \cite{atiyahhirzebruch} should carry over to positive characteristic.\\
For (c), the hard work has been done by Totaro \cite{totaro}, who studied the kernel of the map $MU^{\ast}X\otimes_{MU^{\ast}}\Z \to H^{\ast}(X,\Z)$. The varieties in (c) are the same as in \cite{totaro}, but defined over $k$ instead of $\C$. It is important to note that the variety in (c) is not the first example for the non-injectivity of the map $CH^2X/n\to H^4_{\et}(X;\Z/n)$ for a smooth projective variety $X$ over an algebraically closed field. This question has been discussed by Colliot-Th\'el\`ene in \cite{ct}. Examples have been found by Bloch and Esnault in \cite{blochesnault} and by Schoen in \cite{schoen1}, \cite{schoen2}, \cite{schoen3} and \cite{schoen4}. But in (c) we get new examples of cycles over an algebraically closed field of positive characteristic by methods that are different from the methods of Bloch, Esnault and Schoen.\\
As in the complex case, the factorization of the cycle map and the examples of Corollary \ref{obstruction} are a torsion phenomenon. The map $\hMU_{\et}^{\ast}(X) \to H_{\et}^{\ast}(X;\Zell)$ is the edge map of the top row of the Atiyah-Hirzebruch spectral sequence for \'etale cobordism $E_2=H_{\et}^{\ast}(X;\hMU^{\ast})\Rightarrow \hMU_{\et}^{\ast}X$. All its differentials are torsion. Hence if $H_{\et}^{\ast}(X;\Zell)$ has no torsion, then the $E_2$-term has no torsion and all differentials vanish. In this case the map $\hMU_{\et}^{\ast}(X)\otimes_{\hMU^{\ast}} \Zell \to H_{\et}^{\ast}(X;\Zell)$ is an isomorphism. Moreover, for an arbitrary $X$, the same argument shows that after tensoring with $\Qell$ the map 
$$\hMU_{\et}^{\ast}(X)\otimes_{\hMU^{\ast}} \Qell \to H_{\et}^{\ast}(X;\Zell)\otimes \Qell$$ is always an isomorphism.\\
Finally, Theorem \ref{maintheoremintro} is not implied by the work of Levine and Morel \cite{lm} on algebraic cobordism. After Totaro had written \cite{totaro}, Levine and Morel proved that $\Omega^{\ast}X\otimes_{\Lee^{\ast}} \Z$ is isomorphic to $CH^{\ast}X$ and that this is the universal oriented cohomology with additive formal group law over any field of characteristic zero. So in characteristic zero, Theorem \ref{maintheoremintro} and now the theorem of Totaro are weaker than the result in \cite{lm}. But over a field of positive characteristic, the universality of algebraic cobordism is not known. It is unlikely that Gabber's theorem which we use in the proof of Theorem \ref{maintheoremintro}, suffices for an extension of the work of Levine and Morel in positive characteristic. One would need a more detailed description of the complement of the smooth locus.\\
Before the kickoff, let us resume the outline of the paper. In the next section, we recall the profinite \'etale homotopy functor \cite{artinmazur}, \cite{fried} and \'etale cobordism, first considered in \cite{etalecob}. Since the generalized cycle map can be constructed more generally for \'etale Borel-Moore bordism of not necessarily smooth schemes, we will discuss this theory as well and prove the existence of a cap product pairing. In particular, we show Poincare duality for \'etale bordism for projective smooth varieties over algebraically closed fields. This will allow us to define a fundamental bordism class of a smooth projective variety. The construction of the cycle map and the proof of Theorem \ref{maintheoremintro} will occupy the fourth section. In the last section we will discuss the examples of Atiyah and Hirzebruch and check that the examples in \cite{totaro} of cycles algebraically but not homologically equivalent to zero work over any algebraically closed field of characteristic different from two.

{\bf Acknowledgements}. I would like to thank Christopher Deninger for his steady and motivating interest in this project, Bruno Kahn, Deepam Patel, Michael Joachim, Marc Levine, Mike Hill, Burt Totaro, Jean-Louis Colliot-Th\'el\`ene and Luc Illusie for very helpful discussions, comments and suggestions. I am especially indebted to Mike Hopkins who answered all my questions with patience and shared his ideas with great generosity. Parts of this research have been done during a stay at Havard University. I would like to thank the Mathematics Department and its members for the hospitality and the inspiring atmosphere.
\section{\'Etale realizations and profinite spectra} 
\subsection{The \'etale realization functor}
The starting point for \'etale homotopy theory is the work of Artin and Mazur \cite{artinmazur}. The goal was to define invariants as in algebraic topology for a scheme that depend only on the \'etale topology. Friedlander rigidified their construction by associating to a scheme $X$ a pro-object in the category $\Sh$ of simplicial sets. The construction is in all cases technical and we refer the reader to \cite{fried} for details, in particular for the category $HRR(X)$ of rigid hypercoverings. But let us quickly recall that for a locally noetherian scheme $X$, the \'etale topological type is defined to be the pro-simplicial set $\Et X := \mathrm{Re} \circ \pi :HRR(X) \to \Sh$ sending a rigid hypercovering $U$ of $X$ to the simplicial set of connected components of $U$. If $f: X \to Y$ is a map of locally noetherian schemes, then the strict map $\Et f:\Et X \to \Et Y$ is given by the functor $f^{\ast}:HRR(Y) \to HRR(X)$ of rigid pullbacks and the natural transformation $\Et X \circ f^{\ast} \to \Et Y$. For geometrically unibranched $X$, the pro-fundamental group of $\Et X$ is equal to the profinite \'etale fundamental group of $X$ as a scheme. The cohomology of $\Et X$ as a pro-space equals the \'etale cohomology for locally constant coefficients, see \cite{fried}. \\ 
To get an actual space, one would like to take the inverse limit of the underlying diagram of $\Et X$. But as remarked in \cite{fried}, one would not only lose information but also get the wrong (discrete) invariants. Nevertheless, one can control the loss of information as we explain now.\\
In \cite{profinhom}, we studied a profinite version $\hEt$ of this functor by composing $\Et$ with the completion from pro-$\Sh$ to the category of simplicial profinite sets $\hSh$. This functor is the composite of the completion $\Sh \to \hSh$ and taking the limit of the underlying diagram. We call the objects in $\hSh$ profinite spaces. Its morphisms are simplicial maps that are levelwise continuous. There at least two interesting model structures on $\hSh$. The first one has been studied by Morel in \cite{ensprofin}. For this model structure the cofibrations are the levelwise monomorphisms and the weak equivalences are maps that induce isomorphisms in continuous cohomology with $\Zl$-coefficients for a fixed prime $\ell$. In \cite{profinhom} a different structure has been considered for which the cofibrations are as before but the weak equivalences are maps that induce isomorphisms on profinite fundamental groups and in continuous cohomology for finite local coefficient systems. The model structure of \cite{profinhom} is particularly useful for general profinite completions as it provides a rigid version of the profinite completion functor of Artin and Mazur \cite{artinmazur}. In particular, for the \'etale topological type functor, it provides a suitable control of the limit process if one wants to pass from the pro-object $\Et X$ to an actual simplicial set. For example, the fundamental group of $\hEt X$ as a profinite space is always equal to the \'etale fundamental group of $X$ and the continuous cohomology of $\hEt X$ with profinite local coefficients equals the continuous \'etale cohomology of $X$ defined by Dwyer, Friedlander \cite{dwyfried} and Jannsen \cite{jannsen}.
\subsection{The stable profinite homotopy category}
The target of the cycle map that we are going to construct is a quotient of an $\ell$-adically completed version of cobordism. Therefore, we need the following stabilization of profinite spaces.\\
We denote by $\hSp$ the category of sequences $X_n \in \hShp$ of pointed profinite spaces for $n\geq0$ and maps  $\sigma_n:S^1 \wedge X_n \to X_{n+1}$ in $\hShp$. We call the objects in $\hSp$ profinite spectra. A morphism $f:X \to Y$ of profinite spectra consists of maps $f_n:X_n \to Y_n$ in $\hShp$ for $n\geq0$ such that $\sigma_n(1\wedge f_n)=f_{n+1}\sigma_n$. If $X$ is a pointed profinite space, there is a profinite suspension spectrum $\Sigma^{\infty}X$ given in degree $n$ by the $n$-fold suspension of $X$.\\ 
Let $\ell$ be a fixed prime number. Starting with the $\Zl$-model structure on $\hSh$, whose homotopy category is denoted by $\hHh$, there is a model structure on profinite spectra such that the suspension $S^1\wedge -$ becomes a Quillen equivalence, see \cite{etalecob}, Corollary 16. In other words, $\hSp$ is the stabilization of $\hSh$. We denote the homotopy category of $\hSp$ by $\hSHh$.\\
The completion functor $\compl:\Sh \to \hSh$, which is defined in each dimension by taking the limit over all equaivalence relations with a finite quotient set, and the forgetful functor $|\cdot|: \hSh \to \Sh$ induce levelwise corresponding functors on the category of spectra. When we equip the category of simplicial spectra $\Sp$ with the Bousfield-Friedlander model structure \cite{bousfried}, we get the following adjointness, cf. \cite{etalecob}. 
\begin{prop}\label{quillenpair}
Completion $\compl :\Sp \to \hSp$ preserves weak equivalences and cofibrations.\\
The forgetful functor $|\cdot|: \hSp \to \Sp$ preserves fibrations and weak equivalences between
fibrant objects.\\
In particular, $\compl$ induces a functor on the homotopy categories and the adjoint pair $(\compl , |\cdot|)$ is a Quillen pair of adjoint functors.
\end{prop}
Let $\hMU \in \hSp$ be the completion of the simplicial Thom spectrum representing complex cobordism. For a profinite spectrum $X$, we denote by $\hMU^{n}X$, and call it the $n$th profinite cobordism of $X$, the group of homomorphisms of profinite spectra $\hMU^{n}X:=\Hom_{\hSHh}(X,\hMU\wedge S^n)$. If $X$ is a profinite space, we denote by $\hMU^{n}X$ to be the reduced profinite cobordism of the suspension spectrum $\Sigma^{\infty}(X_+)$, where $X_+$ indicates that we add a disjoint basepoint to $X$.\\
The coefficient ring $\hMU^{\ast}$ can be obtained by taking a fibrant replacement $\hMU^{\ell}$ of $\hMU$ in $\hSp$. It depends of course on $\ell$, in fact, one gets $\hMU^{\ast}=MU^{\ast}\otimes_{\Z}\Zell$, i.e. it is a free polynomial ring $\Zell[x_1,x_2, \ldots]$ with $x_i$ of degree $-2i$. Let $\hMU^{<0}$ denote the ideal of elements in negative degrees. For any pointed profinite space $X$, $\hMU^{\ast}X$ is a $\hMU^{\ast}$-module. The canonical map of profinite spectra $\hMU \to H\Zell$ corresponding to the map of coefficients $\hMU^{\ast} \to \Zell$, which sends all generators $x_i$ to $0$, induces a natural map from $\hMU^{\ast}X$ to the continuous cohomology $H_{\mathrm{cont}}^{\ast}(X;\Zell)$. It vanishes on the submodule $\hMU^{<0}\cdot \hMU^{\ast}X$ and hence induces a natural map
$$\hMU^{\ast}X\otimes_{\hMU^{\ast}} \Zell \to H_{\mathrm{cont}}^{\ast}(X;\Zell).$$  
%Since the quotient $\hMU^{\ast}X/(\hMU^{<0}\cdot \hMU^{\ast}X)$ is isomorphic to the tensor product $\hMU^{\ast}X\otimes_{\hMU^{\ast}} \Zell$, where the map $\hMU^{\ast} \to \Zell$ sends all generators $x_i$, for $i\geq 1$ to $0$, we obtain in fact a map 
We need the following analogue of Totaro's generalization  \cite{totaro}, Theorem 2.1, of Quillen's theorem \cite{quillen}.  
\begin{theorem}\label{quillenstheorem}
Let $X$ be a finite simplicial set. Then the groups $\hMU^nX\otimes_{\hMU^{\ast}} \Zell$ are zero in negative dimensions and equal to $H_{\mathrm{cont}}^0(X;\Zell)$ in dimension $0$. Moreover, the map $\hMU^{\ast}X\otimes_{\hMU^{\ast}} \Zell \to H_{\mathrm{cont}}^{\ast}(X;\Zell)$
is an isomorphism in dimensions $\leq 2$ and injective in dimensions $\leq 4$.
\end{theorem}
\begin{proof}
Let $\hMU^{\ell}$ denote a fibrant replacement of $\hMU$ in $\hSp$. By Proposition \ref{quillenpair} and since $X$ is finite, the forgetful functor $|\cdot|:\hSp \to \Sp$ induces an isomorphism between the continuous $\hMU^{\ell}$-cohomology of $X$ as a profinite space and its $|\hMU^{\ell}|$-cohomology as a simplicial set. Since $X$ is finite, $\Zell$ is torsion-free and $\hMU^{\ast}$ is isomorphic to $MU^{\ast}\otimes_{\Z}\Zell$, we conclude $\hMU^{\ast}(X)\cong MU^{\ast}(X)\otimes_{\Z}\Zell$ for any finite simplicial set $X$. Similarly, $H_{\mathrm{cont}}^{\ast}(X;\Zell)$ is isomorphic to $H^{\ast}(X;\Z)\otimes_{\Z}\Zell$. It follows that for finite $X$, the map $\hMU^{\ast}X\otimes_{\hMU^{\ast}} \Zell \to H_{\mathrm{cont}}^{\ast}(X,\Zell)$ is just the image of the map $MU^{\ast}X\otimes_{MU^{\ast}} \Z \to H^{\ast}(X,\Z)$ under the tensor product with $\Zell$ over $\Z$. Since $\Zell$ is torsion-free, the result follows from Theorem 2.1 of \cite{totaro}. 
\end{proof}
\begin{remark}
Let $BP$ be the Brown-Peterson spectrum at the prime $\ell$ and let $\hBP$ be its profinite completion. For a finite simplicial set $X$, the isomorphism $MU^{\ast}X\otimes_{MU^{\ast}}\Z_{(\ell)}\cong BP^{\ast}X\otimes_{BP^{\ast}}\Z_{(\ell)}$ and the canonical map $\Z_{(\ell)} \to \Zell$ induce an isomorphism 
$$\hMU^{\ast}X\otimes_{\hMU^{\ast}}\Zell \cong \hBP^{\ast}X\otimes_{\hBP^{\ast}}\Zell.$$
Thus the whole game could have been played using $\hBP$ instead of $\hMU$. But we stick to $\hMU$.% since below we use the image of the orientation of $MGL$ under \'etale realization. 
\end{remark}
One of the most important properties of $\hSHh$ is that it is the target category of the extension of the \'etale realization functor to motivic spectra. Recall that the \'etale realization functor has been extended to motivic spaces by Isaksen in \cite{a1real}. So, for example, we may talk about the \'etale homotopy type of the Thom space of the normal bundle of a regular embedding etc. The following extension has been proved in \cite{etalecob}; in fact one can obtain a more general result over an arbitrary base field.
\begin{theorem}
Let $k$ be an algebraically closed field of characteristic $\neq \ell$. The \'etale realization functor of simplicial presheaves has a natural extension to the stable homotopy category of motivic $\Pro^1$-spectra $\hEt: \SHh(k) \to \hSHh$. The image of the spectrum $MGL$ representing algebraic cobordism is isomorphic to $\hMU$ in $\hSHh$.
\end{theorem}
\section{\'Etale topological bordism}
For the rest of this paper we assume that $k$ is an algebraically closed field. Let $X$ be a scheme over $k$. Let $\ell$ be a prime different from the characteristic of $k$. We equip $\hSh$ and $\hSp$ with the $\Zl$-model structures of the previuos section. We will define \'etale topological bordism and cobordism groups of $X$. Let us start with the latter one. 
\subsection{\'Etale cobordism}
We define \'etale topological cobordism to be the generalized cohomology theory represented by $\hMU$ in $\hSHh$ applied to $\hEt X$, i.e.
$$\hMU_{\et}^{n}(X):=\Hom_{\hSHh}(\Sigma^{\infty}(\hEt X),\hMU\wedge S^n).$$
If $X$ is not equipped with a specified basepoint, we will denote by $\hMU_{\et}^{n}(X)$ the reduced \'etale cobordism of $\hEt X_+$.\\
We denote by $\Sm_k$ the category of quasiprojective smooth schemes of finite type over $k$. We recall from \cite{etalecob} that $\ell$-adic \'etale topological cobordism is an oriented cohomology theory on $\Sm_k$ in the sense of \cite{panin}, Definition 2.1. We will outline the proof below for completeness.\\
The first thing to do, is to equip \'etale cobordism with an orientation. There is the following canonical choice. Let $MGL$ be the motivic spectrum representing algebraic cobordism and let  $x_{MGL}:\Pro^{\infty} \to MGL \wedge \Pro^1$ be its orientation, see e.g. \cite{panin}. We define the orientation $x_{\hMU}:=\hEt(x_{MGL}): \hEt \Pro_k^{\infty} \to \hMU \wedge S^2$ of \'etale cobordism to be the image of the orientation of algebraic cobordism under $\hEt$. Since the isomorphism $\hEt MGL\cong \hMU$ is constructed using a lifting to characteristic zero, the appropriate identifications show that this orientation corresponds to the image under completion of the canonical orientation of $MU$. 
\begin{theorem}\label{orientation}
Let $k$ be an algebraically closed field and let $\ell$ be a prime different from the characteristic of the base field $k$. With the above orientation, $\ell$-adic \'etale topological cobordism is an oriented cohomology theory on $\Sm_k$. In particular, for every projective morphism $f:X \to Y$ of relative codimension $d$ in $\Sm_k$, there are pushforward maps $f_{\ast}:\hMU_{\et}^{2\ast}(X) \to \hMU_{\et}^{2\ast +2d}(Y)$. 
%The map $\theta:\hMU_{\et}^{\ast}(X) \to H_{\et}^{\ast}(X;\Zell)$ is the canonical map of oriented cohomology theories induced by the orientation map of spectra $\hMU \to H\Zell$.
\end{theorem}
\begin{proof}
That \'etale cobordism is a ring cohomology theory follows immediately from the properties of $MU$. The $\A^1$-invariance follows from the fact that, since we consider the $\Zl$-model structure, $\hEt (X\times \A^1) \to \hEt X$ is a weak equivalence in $\hSh$. This would not be true in general if we had not completed away from the characteristic of the base field. To check excision, let $e:(X',U') \to (X,U)$ be a morphism of pairs of schemes in $\Sm_k$ such that $e$ is \'etale and for $Z=X-U$, $Z'=X' -U'$ one has $e^{-1}(Z)=Z'$ and $e:Z' \to Z$ is an isomorphism. By \cite{milne} III, Proposition 1.27, we know that the morphism $e$ induces an isomorphism in \'etale cohomology $H^{\ast}(\hEt(X)/\hEt(U);\Zl) \cong H^{\ast}(\hEt(X')/\hEt(U');\Zl)$. Hence the map $\hEt(X')/\hEt(U') \to \hEt(X)/\hEt(U)$ is an isomorphism in $\hHhp$. Therefore, it induces the desired isomorphism 
$\hMU^{\ast}(\hEt(X)/\hEt(U)) \cong \hMU^{\ast}(\hEt(X')/\hEt(U'))$.\\   
It remains to check that \'etale cobordism is an oriented theory, i.e. that it has Chern classes. Let $E \to X$ be a vector bundle of rank $n$ over $X$ and let $\Oh(1)$ be the canonical quotient line bundle over its projective bundle $\Pro(E)$. It determines a morphism $\Pro(E) \to \Pro^N_k$ for some sufficiently large $N$. Together with the orientation map $x_{\hMU}$ we get an element $\xi \in \hMU^2_{\et}(\Pro(E))$. This induces a projective bundle formula for \'etale cobordism, i.e.  $\hMU_{\et}^{\ast}(\Pro(E))$ is a free $\hMU_{\et}^{\ast}(X)$-module with basis $(1, \xi, \xi^2, \ldots, \xi^{n-1})$. For, 
the projective bundle formula for \'etale cohomology implies that the canonical map $\hEt (X) \times \hEt (\Pro^n_k) \to \hEt (X\times \Pro^n_k)$ is a weak equivalence in $\hSh$; this implies a projective bundle formula locally and a Mayer-Vietoris argument shows that the formula holds globally. Chern classes for $E$ are then defined in the well known way. The existence and uniqueness of pushforward maps follows from \cite{panin2}.
\end{proof}

The next proposition will allow us to apply the completed version of Quillen's Theorem \ref{quillenstheorem}.
\begin{prop}\label{finiteness}
Let $X$ be an $n$-dimensional scheme of finite type over an algebraically closed field. Let $\ell$ be a prime different from the characteristic of $k$. Then $\hEt X$ has the homotopy type of a finite simplicial set in $\hSh$ with respect to the $\Zl$-model structure.
\end{prop}
\begin{proof}
It suffices to remark that the \'etale cohomology groups $H_{\et}^i(X;\Zl)$ are finite for every $i$ and vanish for $i\geq 2n+1$. 
\end{proof}
\subsection{\'Etale Borel-Moore bordism}
The cycle map that will be constructed in the next section will in fact be a map from algebraic cycles to a quotient of \'etale bordism. We define \'etale bordism of a pointed scheme $X$ to be the profinite homology theory represented by $\hMU$ in $\hSHh$ applied to $\hEt X$, i.e.
$$\hMU^{\et}_{n}(X):=\Hom_{\hSHh}(S^n,\hMU\wedge \Sigma^{\infty}(\hEt X)).$$
If $X$ is not equipped with a specified basepoint, we will denote by $\hMU^{\et}_{n}(X)$ the reduced \'etale bordism of $\hEt X_+$.\\
We need also a more refined version of bordism. The cycle map that will be defined in the next section takes values in Borel-Moore \'etale bordism rather than cobordism. We will define it in the same way as Friedlander defined Borel-Moore \'etale homology in \cite{fried}, Proposition 17.2. If $X$ is a scheme of finite type over $k$, let $\oX$ be a compactification of $X$.  Such a compactification always exists by Nagata's Theorem \cite{nagata}. Then we denote by 
$$\hMU_{\ast}^{BM,\et}(X):=\hMU_{\ast}(\hEt \oX,\hEt (\oX-X))$$
the \'etale topological Borel-Moore bordism of $X$. 
\begin{lemma}\label{compactification}
\'Etale Borel-Moore bordism does not depend on the choice of a compactification. A proper map $f:X \to Y$ induces a pushforward map $$f_{\ast}:\hMU_{\ast}^{BM,\et}(X) \to \hMU_{\ast}^{BM,\et}(Y).$$  
\end{lemma}
\begin{proof}
Let $j_1:X \into \oX_1$ and $j_2:X\into \oX_2$ be two compactifications. The argument in the proof of Proposition 17.2 in \cite{fried} shows that we can assume that $j_1$ maps to $j_2$ via a proper map $\bar{f}:\oX_2 \to \oX_1$ restricting to an isomorphism $j_2(X) \stackrel{\sim}{\to} j_1(X)$ and a map $Y_2:=\oX_2- j_2(X) \to \oX_1 -j_1(X)=:Y_1$. Moreover, Proposition 17.2 of \cite{fried} implies that $\hEt \bar{f}$ induces an isomorphism $\hEt \oX_2/\hEt Y_2 \cong \hEt \oX_1/\hEt Y_1$ in $\hHh$. Thus $\hEt \bar{f}$ induces an isomorphism in any homology theory.\\
Now let $f$ be a proper morphism. We argue as in \cite{fried}, 17.2. Let $j_0:X \into \oX_1$ and $j_1:Y\into \oY$ be compactifications, then $f$ induces $\bar{f}:(\oX_2,\oX_2-j_2(X)) \to (\oY,\oY-j_1(Y))$, where $j_2:X\into \oX_2$ is the compactification defined by $\Delta:X \to \oX \times \oY$. Hence $\bar{f}$ induces a pushforward map in \'etale bordism for these pairs.
\end{proof}

By the work of Cox \cite{cox1}, \cite{cox2} and Friedlander \cite{fried}, for any closed immersion $X\into Y$ of locally noetherian schemes, there is an \'etale tubular neighborhood of $X$ in $Y$ denoted by $T_{Y/X}$. Let $\psi_{X}$ be the functor that sends an \'etale covering $U\to Y$ to the union of those connected components $U^{\alpha}$ of $U$ with $U^{\alpha}\times_XY$ is nonempty. Then $T_{Y/X}$ is defined to be the pro-object of simplicial schemes
$$T_{Y/X}=\{\psi_Y(U_{\cdot}); U_{\cdot}\in HRR(Y)\}.$$
So $T_{Y/X}$ collects those hypercoverings of $Y$ that intersect the hypercoverings of $X$. We get its \'etale topological type $\Et T_{Y/X}$ by applying the connected component functor to each $\psi_Y(U_{\cdot})$. It is a pro-simplicial set and we let $\hEt T_{Y/X}$ be the associated profinite space. An important fact is that $\hEt T_{Y/X}$ is weakly equivalent to $\hEt X$ and that there are weak equivalences 
\begin{equation}\label{tubularexcision}
(\hEt Y,\hEt X) \simeq (\hEt (T_{Y/Y}\times_Y (Y-X)), \hEt (T_{Y/X}\times_Y (Y-X))) %\cong (\hEt Y-X, \hEt (T_{Y/X}\times_Y (Y-X)))
\end{equation}
of pairs of profinite spaces, see \cite{fried} \S 15. This excision property (\ref{tubularexcision}) of tubular neighborhoods allows us to define a cap product for any \'etale homology theory as in \cite{fried}.
\begin{prop}\label{capproduct}
Let $i:X \into Y$ be a closed immersion of schemes over $k$. For each $n\geq p \geq 0$, there is a natural cap product pairing
$$\hMU_{\et}^{p}(Y,Y-X)\otimes \hMU^{BM,\et}_{n}(Y) \to \hMU^{BM,\et}_{n-p}(X).$$
\end{prop}
\begin{proof}
The cap product can be defined as in \cite{fried}, Proposition 17.4, using \'etale tubular neighborhoods and the usual construction of cap products for generalized relative homology theories in \cite{adams}. Consider the sequences of embeddings
$$X\stackrel{i}{\to}Y\stackrel{j}{\to}\bar{Y}, X \stackrel{j'}{\to} W \stackrel{i'}{\to} \bar{Y}$$
where $j$ is a compactification of $Y$, $i'$ is the closed immersion of $W:=\bar{Y}-(Y-X)$ into $\bar{Y}$ and $j'$ is the associated open immersion. For convenience, we set $V:=W-X=\bar{Y}-Y$. The above excision formula (\ref{tubularexcision}) yields the identifications of pairs in $\hHh$:
$$(\hEt Y,\hEt Y-X) \cong (\hEt (T_{\bar{Y}/W}\times_{\bar{Y}} Y), \hEt (T_{\bar{Y}/W}\times_{\bar{Y}} (Y-X))),$$
$$(\hEt W,\hEt V) \cong (\hEt (T_{\bar{Y}/W}\times_{\bar{Y}} Y), \hEt (T_{\bar{Y}/V}\times_{\bar{Y}} Y)),$$
$$(\hEt \bar{Y},\hEt (\bar{Y}-Y)) \cong (\hEt (T_{\bar{Y}/W}\times_{\bar{Y}} Y), \hEt (T_{\bar{Y}/W}\times_{\bar{Y}} (Y-X)) \cup \hEt (T_{\bar{Y}/V}\times_{\bar{Y}} Y)).$$
Moreover, we know $(\hEt W,\hEt V)\cong (\hEt \bar{X},\hEt (\bar{X}-X))$ in $\hHh$ for a compactification $\bar{X}$ of $X$, see \cite{fried}, proof of Proposition 17.4. Finally, we apply the usual cap product construction of \cite{adams}, Part III \S 9, via the slant product, to  
$$\hEt (T_{\bar{Y}/W}\times_{\bar{Y}} (Y-X)) \into \hEt (T_{\bar{Y}/W}\times_{\bar{Y}} Y)~\mathrm{and}$$
$$ \hEt (T_{\bar{Y}/V}\times_{\bar{Y}} Y) \into \hEt (T_{\bar{Y}/W}\times_{\bar{Y}} Y).$$ 
\end{proof}
\subsection{Poincar\'e duality}
Both theories, \'etale bordism and cobordism, are closely related by the following Poincar\'e duality theorem for smooth projective varieties over $k$, where variety means a reduced and irreducible   scheme that is spearated and of finite type over $k$. 
\begin{theorem}\label{duality}
Let $X$ be a smooth projective variety over $k$ of dimension $n$. Then there is a natural Poincar\'e duality isomorphism
$$D^X:\hMU_{p}^{\et}(X) \stackrel{\cong}{\to} \hMU_{\et}^{2n-p}(X).$$
We define the fundamental class $[X] \in \hMU_{2n}^{\et}(X)$ of $X$ to be the inverse image of $1 \in \hMU_{\et}^0(X)$ under the above duality map $D^X$. The inverse $D_X$ of $D^X$ is given by the cap-product with $[X]$.
\end{theorem}
Before we prove this theorem, we make the following convention.
\begin{convention}\label{choice}
Let $k$ be our algebraically closed base field and $\ell$ the chosen prime different from the characteristic of $k$. For a positive integer $n\geq 1$, let $\mu_{\ell^n}(k)$ denote the group of $\ell^n$th roots of unities in $k$. For the rest of this paper, we fix a choice of a compatible system of $\ell^n$th roots of unities in $k$ for all $n \geq 1$ and use the induced isomorphism to make the  identification $\Zell\cong \Zell(1)=\lim_n \mu_{\ell^n}(k)$. 
\end{convention}
\begin{proof}
We use the classical construction for this duality map in topology and define the map $D^X$ via the slant product for generalized homology theories, see \cite{adams}. The orientation map $\hEt \Pro^{\infty} \to \hMU$ induces an orientation for every vector bundle in $\Sm_k$. In particular, for the normal bundle $N$ of the embedding of the diagonal $\Delta(X)\into X\times X$, which is isomorphic to the tangent bundle of $X$, we get an induced map map $\omega:\hEt \mathrm{Th}(N) \to \Sigma^{2n}\hMU$ in $\hSHh$. By homotopy purity, we know $X\times X/X\times X - \Delta(X)$ is $\A^1$-equivalent to $\mathrm{Th}(N)$, see \cite{hu} \S 3. This map induces a map in $\hSHh$
$$\hEt (X \times X) \to \hEt (X \times X/X \times X - \Delta(X)) \to \Sigma^{2n}\hMU.$$
By abuse of notation, we will also denote this composed map by $\omega$. Now given an element $g \in \hMU_{p}^{\et}(X)$, represented by a map $g:S^p \to \hMU \wedge \hEt X$, the slant product $\omega/g \in \hMU_{\et}^{2d-p}(X)$ is defined as the following composite of maps in $\hSHh$, where we omit to denote the twists,
$$S^p \wedge \hEt X \stackrel{g\wedge 1}{\to} \hMU \wedge \hEt X \wedge \hEt X \stackrel{1 \wedge \omega}{\to} \hMU \wedge \Sigma^{2n}\hMU \stackrel{\mu}{\to} \Sigma^{2n}\hMU$$
where $\mu:\hMU \wedge \hMU \to \hMU$ is the multiplication map of the ring spectrum $\hMU$. We define the map $D^X$ by sending $g$ to $D^X(g):=\omega/g$.\\
In order to prove that $D^X$ is an isomorphism, we apply the following diagram of Atiyah-Hirzebruch spectral sequences
$$\xymatrix{
H_{\et}^{2n-s}(X; \hMU_t)\ar[d]_{D^X} \ar@{=>}[r] & \hMU_{\et}^{2n-s-t}(X) \ar[d]^{D^X}\\
H^{\et}_{s}(X; \hMU_t) \ar@{=>}[r] & \hMU^{\et}_{s+t}(X)}$$
where $H^{\et}_{s}(X; \hMU_t)$ denotes continuous \'etale homology as defined in \cite{profinhom} and $H_{\et}^{2n-s}(X; \hMU_t)$ denotes continuous cohomology of $\hEt X$ as defined in \cite{profinhom}. As we remarked above, $H_{\et}^{\ast}(X; \Zell)$ coincides with Dwyer-Friedlander's \cite{dwyfried} and Jannsen's \cite{jannsen} continuous cohomology of $X$. Since $X$ is smooth, these groups are equal to the usual $\ell$-adic cohomology $\lim_{\nu} H_{\et}^{\ast}(X; \Z/\ell^{\nu})$. Since $MU_t$ is a free abelian group for every $t$, the groups $H_{\et}^{2n-s}(X; \hMU_t)$ are isomorphic to $\lim_{\nu} H_{\et}^{\ast}(X; \Z/\ell^{\nu})\otimes MU_t$. The construction of $D^X$ and the fact that the map $\hMU \to H\Zell$ is a map of oriented spectra show that the same construction of $D^X$ in homology yields a natural map between the two spectral sequences. Since $k$ is algebraically closed, the $E^2$-terms vanish except for $0 \leq s \leq 2n$ and the spectral sequences converge. Finally, the Poincar\'e duality of \cite{fried}, Theorem 17.6, %where the inverse map $D_X$ is used, 
shows that the $E^2$-terms of the spectral sequences are isomorphic. This proves that $D^X$ on bordism is an isomorphism. That the inverse $D_X$ is given by the cap product with $[X]$ now follows just as in \cite{adams}, Part III, Proposition 10.16.
\end{proof}  
\begin{prop}\label{pullback}
Let $i:X\into Y$ be a closed embedding of smooth schemes of codimension $d$. Let $\omega_{X/Y}\in \hMU^{2d}_{\et}(Y,Y-X)$ be the element corresponding to the Thom class of the normal bundle $N$ under the purity isomorphism $\hMU^{2d}_{\et}(Y,Y-X)\cong \hMU^{2d}_{\et}(N,N-X)$, cf. \cite{mv} Theorem 3.2.23. We call $\omega_{X/Y}$ the orientation class of the embedding. Then the cap product with $\omega_{X/Y}$ defines a natural pullback map 
$$i^{\ast}:\hMU^{BM,\et}_{j}(Y) \to \hMU^{BM,\et}_{j-2d}(X)$$
on \'etale Borel-Moore bordism. If $X$ and $Y$ are smooth projective varieties, this pullback coincides with the induced pullback from \'etale cobordism by Poincar\'e duality.
\end{prop}
\begin{proof}
The only assertion in this proposition is the one for the case that $i$ is a closed embedding of smooth projective varieties. For $X$ and $Y$ projective, the pairs used to define a cap product pairing are given by the \'etale homotopy types of the inclusions
$$T_{Y/X}\times_Y(Y-X) \into T_{Y/X}\times_Y Y~\mathrm{and}$$
$$T_{Y/\emptyset}\times_Y Y \into T_{Y/X}\times_Y Y.$$
The cap product pairing is translated by $D^Y$ and $D^X$ respectively into the corresponding cup product pairing. That this cup product pairing, induced by multiplication with the class $\omega_{X/Y}$, coincides with the pullback map for cobordism now follows from the Thom isomorphism theorem for an oriented cohomology theory. 
\end{proof}
\section{The cycle class map}
The goal of this section is to prove the following analogue of Theorem 3.1 of \cite{totaro} for a base field of arbitrary characteristic. By Poincar\'e duality, this proves Theorem \ref{maintheoremintro} of the introduction for a smooth projective variety $X$. 
\begin{theorem}\label{maintheorem}
Let $k$ be an algebraically closed field and let $X$ be a quasiprojective scheme of finite type over $k$. Let $\ell$ be a prime different from $p$. There is a natural map $\clMU:Z_iX \to \hMU^{BM,\et}_{2i}(X)\otimes_{\hMU_{\ast}} \Zell$ from the group of algebraic cycles of dimenson $i$ on $X$ that vanishes on cycles algebraically equivalent to zero such that the composition
$$Z^{\mathrm{alg}}_i X \stackrel{\clMU}{\longrightarrow} \hMU^{BM,\et}_{2i}(X)\otimes_{\hMU_{\ast}} \Zell \stackrel{\theta}{\longrightarrow} H^{BM,\et}_{2i}(X;\Zell)$$
is the cycle class map $\clH$ to \'etale homology of \cite{fried} Proposition 17.4, where $Z^{\mathrm{alg}}_iX$ denotes the group of cycles modulo algebraic equivalence. This map $\clMU$ is natural with respect to projective morphisms.
\end{theorem}
We remind the reader that in the above statement and in the following prove we stick to Convention \ref{choice} to use the fixed isomorphism to identify $\Zell=\Zell(1)$. Before we prove the theorem, let us start with some comments. If $k=\C$ is the field of complex numbers, the cycle map of \cite{totaro}, Theorem 3.1, sends a closed irreducible subscheme $Z\subset X$ of dimension $i$ is sent to the class $[\tilde{Z} \to X]\in MU^{BM}_{2i}X\otimes_{MU_{\ast}}\Z$, where $\tilde{Z} \to Z$ denotes a resolution of singularities. This is also the unique  map induced by the universality of Chow groups among oriented cohomology theories whose formal group law is additive, see Remark 1.2.21 and the proof of Theorem 4.5.1 in \cite{lm}.\\  
In both papers \cite{totaro} and \cite{lm}, maps for integral coefficients are constructed using resolution of singularities for fields of characteristic zero, a technique that is so far not available over a base of positive characteristic. The best known replacement is the work of de Jong \cite{dejong} on alterations and its improvement by Gabber. An alteration of a Noehterian integral scheme $X$ over a field $k$ is a dominant proper morphism $\pi:X' \to X$ from an integral scheme $X'$ to $X$ with $\dim X'=\dim X$. The map $\pi$ is finite and flat over a nonempty open subset of $X$. In \cite{dejong} de Jong proved that for any variety $X$ and any proper closed subset $Z\subset X$, there is a regular alteration $\pi:X' \to X$ such that $\pi^{-1}(Z)$ is the support of a strict normal crossings divisor in some regular projective variety $\bar{X}'$. Moreover, if $k$ is a perfect field, then $\bar{X}'$ and hence also $X'$ are smooth over $k$ and $p$ is generically \'etale. This is weaker than a resolution of singularities of $X$, since alterations allow $k(X')$ to be a finite extension of the function field $k(X)$ whereas a resolution of singularities would require $k(X')=k(X)$.\\ 
In order to construct a well defined cycle map as in the theorem an arbitrary smooth alteration would not suffice, since we would not be able to show that two different alterations of $X$ define the same element in the quotient $\hMU^{\et}_{\ast}(X)\otimes_{\hMU_{\ast}} \Zell$. The problem is that there used to be no control on the degree of the alteration. But recently, Gabber improved de Jong's result further by showing that there is an alteration with some control on the degree of the extension $k(X')/k(X)$. To be more precise, Gabber proved the following result, cf. \cite{gabber}; see also \cite{illusie} for a more detailed account. 
\begin{theorem}\label{gabber}{\rm (Gabber)}
Let $X$ be a separated scheme of finite type over a perfect field $k$, $Z\subset X$ a nowhere dense closed subset, $\ell$ a prime $\neq \car (k)$. Then there exists an alteration $\pi:X' \to X$ of degree prime to $\ell$ with $X'$ smooth and quasiprojective over $k$ and $\pi^{-1}(Z)$ the support of a strict normal crossings divisor. 
\end{theorem}
We will now prove Theorem \ref{maintheorem} using Gabber's result. Let $X$ and $k$ be as in Theorem \ref{maintheorem} and let $Z\subset X$ be an irreducible subvariety of $X$ of dimension $i$. Since $X$ is quasiprojective, it has a projective compactification $\bar{X}$. The closure $\bar{Z}$ of $Z$ in $\bar{X}$ is a projective compactification of $Z$. Now let $\bar{Z}_0\subset \bar{Z}$ be the singular locus of $\bar{Z}$. We apply Theorem \ref{gabber} to the pair $\bar{Z}_0 \subset \bar{Z}$. We get an alteration $\pi:Z' \to \bar{Z}$ of degree $d$ prime to $\ell$ with $Z'$ smooth and, since $\pi$ is quasiprojective and proper, $Z'$ is also projective over $k$ 
$$\xymatrix{
  &  Z' \ar[d]_{\pi}\\
Z \ar[r] \ar[d] & \bar{Z} \ar[d]\\
X\ar[r] & \bar{X}.}$$
Hence $Z'$ has a fundamental class in $\hMU^{\et}_{2i}(Z')\otimes_{\hMU_{\ast}} \Zell$ by Theorem \ref{duality}. Since $(d,\ell)=1$, $d$ is invertible in $\Zell$. We send the cycle $Z\subset X$ to the pushforward of the fundamental class of $Z'$ under the projective map $Z' \to \bar{X}$ divided by the degree:  
$$\clMU(Z):=\frac{1}{d} [Z'\stackrel{\pi}{\to} \bar{Z} \subset \bar{X}] \in \hMU^{BM,\et}_{2i}(X)\otimes_{MU_{\ast}} \Zell.$$ 
Here we use the notation of \cite{lm} and write $[Y\to X]$ for the class $f_{\ast}([Y])$ in $\hMU^{BM,\et}_{2n}(X)$, for a projective morphism $f:Y\to X$ of schemes over $k$ and $Y$ a smooth and projective variety over $k$ of dimension $n$ with fundamental class $[Y]$.\\ 
We have to show that this definition is independent of the choice of $Z'$. %The idea is the same as in \cite{totaro} and \cite{lm}, but we have to check that we are allowed to use this kind of alteration instead of resolution of singularities. 
We have already shown that Borel-Moore bordism is independent of the choice of compactification. It remains to show the independence of the choice of alteration. Let $\pi_1:Z'_1 \to \ovZ$ and $\pi_2:Z'_2 \to \ovZ$ be two smooth $\ell$-primary alterations of $Z$ of degree $d_1$ and $d_2$ respectively. The point is that the difference between the two classes corresponding to $\pi_1$ and $\pi_2$ lies in the subgroup $\hMU_{\ast>0}^{\et}(X)$, hence it vanishes in the quotient $\hMU^{BM,\et}_{2i}(X)\otimes_{\hMU_{\ast}} \Zell$. Let us check this. There is a third $\ell$-primary alteration $Z'_3$ that dominates both $Z'_1$ and $Z'_2$. For this, it suffices to construct a smooth $\ell$-primary alteration $Z'_3$ of the fibre product $Z'_1\times_{\ovZ} Z'_2$:
\begin{equation}\label{alterationsquare}
\xymatrix{
Z'_3 \ar[r] & Z'_1\times_{\ovZ} Z'_2 \ar[r] \ar[d] & Z'_1 \ar[d] \\
 & Z'_2 \ar[r] & \ovZ.}\end{equation}
Since $\pi_1$ and $\pi_2$ are generically \'etale, their properties and degrees are preserved under their mutual base change, i.e. the map $Z_1\times_Z Z_2 \to Z_1$ is proper dominant and generically \'etale of degree $d_2$, the map $Z_1\times_{\ovZ} Z_2 \to Z_2$ is also proper, dominant and generically \'etale of degree $d_1$. Hence $\pi_3:Z'_3\to \ovZ$ is a smooth alteration of $\ovZ$. Moreover, since all alterations were chosen of degree prime to $\ell$, we conclude that the degree $d_3$ of $\pi_3$ is prime to $\ell$ and $\pi_3$ is a smooth $\ell$-primary alteration refining $\pi_1$ and $\pi_2$.\\  
Let $e_1$ and $e_2$ be the degrees of the maps $Z'_3 \to Z'_1$ and $Z'_3 \to Z'_2$. They satisfy the equality $d_3=d_1e_1=d_2e_2$.  Now the class of a map $f:Y\to X$ of smooth varieties of the same dimension $n$ is equal to the class of the identity $X \to X$ multiplied by the degree of $f$ in $H_{2n}^{\et}(X;\Zell)$. The isomorphism $\hMU^{\et}_{2n}(X)\otimes_{\hMU_{\ast}} \Zell \cong H_{2n}^{\et}(X;\Zell)$ of Theorems \ref{quillenstheorem} and \ref{duality} shows that this relation also holds in the quotient of \'etale bordism. Hence we get the two equalities
$$[Z'_3 \to Z'_1]=e_1[Z'_1 \to Z'_1] ~ \mathrm{in}~\hMU^{\et}_{2i}(Z'_1)\otimes_{\hMU^{\ast}} \Zell$$
and 
$$[Z'_3 \to Z'_2]=e_2[Z'_2 \to Z'_2] ~ \mathrm{in}~\hMU^{\et}_{2i}(Z'_2)\otimes_{\hMU^{\ast}} \Zell$$
and as a consequence also
$$\frac{1}{d_1}[Z'_1 \to \oX]=\frac{1}{d_3}[Z'_3 \to \oX]=\frac{1}{d}_2[Z'_2 \to \oX]$$
in $\hMU^{BM,\et}_{2i}(X)\otimes_{\hMU_{\ast}} \Zell$. Hence $\pi_1$ and $\pi_2$ define the same element, called the class of $Z$ in $\hMU^{BM,\et}_{2i}(X)\otimes_{MU_{\ast}} \Zell$. Finally, we extend this map to arbitrary cycles by linearity.\\ 
We have to check that this cycle map $\clMU$ induces the cycle map 
$$\clH: Z_iX \to H_{2i}^{BM,\et}(X;\Zell)$$ 
to \'etale Borel-Moore homology, see e.g. \cite{fried} Proposition 17.4 for a definition. The map $\clH$ sends the class of an irreducible subvariety $Z\subset X$ of dimension $i$ to the image of the fundamental class of $Z$ under the pushforward $H_{2i}^{BM,\et}(Z;\Zell) \to H_{2i}^{BM,\et}(X;\Zell)$. Since the morphism $\theta$ from bordism to homology is a transformation of oriented homology theories, it sends the fundamental classes to fundamental classes and is compatible with pushforwards. Hence to prove the statement, it suffices to observe as before that, if $f:Z' \to Z$ is a finite morphism of degree $d$ of schemes of the same dimension $i$ over $k$, the pushforward $f_{\ast}$ sends the class $[Z']$ to $d[Z]$ in $H_{2i}^{BM,\et}(Z;\Zell)$. This implies $\clH=\theta \circ \clMU$ on the level of cycles.\\  
%it suffices to show that the fundamental class $[Z]_{MU}\in \hMU^{BM,\et}_{2i}(Z)\otimes_{MU_{\ast}} \Zell$ of $Z$ previously defined is sent to the fundamental class $[Z]_H \in H_{2i}^{BM\et}(Z,\Zell)$. So let $\pi:Z' \to Z$ be a smooth $\ell$-primary alteration of $\ovZ$ and let $d$ be the degree of $\pi$. Then $[Z]_{MU}$ is defined as $\frac{1}{d}\pi_{\ast}([Z'])$. The map $\theta_Z:\hMU^{BM,\et}_{2i}(Z)\otimes_{MU_{\ast}} \Zell \to H_{2i}^{BM\et}(Z,\Zell)$ sends the fundamental class $[Z]_{MU}$ to $\frac{1}{d}\pi_{\ast}(\theta_{Z'}([Z']_{MU})) \in H_{2i}^{BM\et}(Z,\Zell)$. So it remains to check $\frac{1}{d}\pi_{\ast}(\theta_{Z'}([Z']_{MU}))=[Z]_H$ in $H_{2i}^{BM\et}(Z,\Zell)$. On the smooth projective variety $Z'$ we know $\theta_{Z'}([Z']_{MU})=[Z']_H$ in $H_{2i}^{BM\et}(Z',\Zell)$ by Theorem \ref{quillenstheorem}. Hence we are reduced to show $\frac{1}{d}\pi_{\ast}([Z']_{H})=[Z]_H$ in $H_{2i}^{BM\et}(Z,\Zell)$. But since $\pi_{\ast}$ is just multiplication by the $d$, i.e. $[\pi:Z' \to Z]=d[Z\to Z]$ in $H_{2i}^{BM\et}(Z,\Zell)$, this equality is obvious. 
%
\begin{lemma}\label{pushforwards}
The map $\clMU$ commutes with pushforwards along projective morphisms.
\end{lemma} 
\begin{proof}
We argue basically as in \cite{totaro}. The only point to check is that inserting degrees does not change the argument. Let $f:X\to Y$ be a projective morphism. For a closed subvariety $Z\subset X$ of dimension $i$, $f(Z)$ is a closed subvariety of $Y$ and the pushforward on $Z_iX$ is defined by
\[
f_{\ast}(Z) = \left \{ \begin{array}{l@{\quad \quad}l}
                                          \mathrm{deg}\,(f: Z \to f(Z))f(Z) & \mathrm{if}~\dim f(Z)=\dim Z\\
                                          0 & \mathrm{if} ~ \dim f(Z)<\dim Z. 
                                          \end{array} \right. 
\]  
Let $d$ denote the degree of the map $f:Z\to f(Z)$. If $\dim f(Z)=\dim Z$, let $\pi_2:Z'_2\to f(Z)$ be an $\ell$-primary smooth alteration of degree $d_2$ of $f(\ovZ)$ and let $\pi_1:Z'_1 \to \ovZ$ an $\ell$-primary alteration over $\pi_2$ of degree $d_1$ constructed as in (\ref{alterationsquare}) such that we get a commutative diagram
$$\xymatrix{
Z'_1 \ar[d] \ar[r] & \ovZ \ar[d] \ar[r] & \oX\ar[d]_{\bar{f}} \\
Z'_2 \ar[r] & f(\ovZ) \ar[r] & \oY.}$$ 
Let $e$ be the degree of $Z'_1 \to Z'_2$. The diagram shows $d_1d=d_2e$. From the isomorphism $\hMU_{2i}^{BM,\et}(Z'_2)\otimes_{\hMU_{\ast}} \Zell \cong H^{BM,\et}_{2i}(Z'_2;\Zell)$ and the relations in $H^{BM,\et}_{2i}(Z'_2;\Zell)$ we get 
$$[Z'_1 \to Z'_2]=e[Z'_2 \to Z'_2]~\mathrm{in}~\hMU^{BM,\et}_{2i}(Z'_2)\otimes_{\hMU_{\ast}} \Zell.$$
The image under the pushforward 
$$f_{\ast}: \hMU_{2i}^{BM,\et}(X)\otimes_{\hMU_{\ast}} \Zell \to \hMU_{2i}^{BM,\et}(Y)\otimes_{\hMU_{\ast}} \Zell$$ 
of the class $\clMU(Z)$ is by definition the class of the map $Z'_1 \to \oY$. Now since the above diagram commutes, this map factors through $Z'_2$. Hence the above equalities imply the following identifications in $\hMU_{2i}^{BM,\et}(Y)\otimes_{\hMU_{\ast}} \Zell$ 
$$f_{\ast}(\clMU(Z))=\frac{1}{d_1} [Z'_1\to \oX \stackrel{\bar{f}}{\to} \oY] = e\frac{1}{d_1} [Z'_2 \to f(\ovZ) \to \oY]=d\, \clMU(f(Z)).$$
This proves the lemma for the case $\dim f(Z)=\dim Z$.\\
If $\dim f(Z) < \dim Z$, then the argument is similar using the fact that the class of a projective map $f:X\to Y$ of smooth schemes with $\dim X>\dim Y$ is $0$ in $\hMU^{BM,\et}_{2i}(Y)\otimes_{\hMU_{\ast}}\Zell$ by Quillen's Theorem \ref{quillenstheorem}.
\end{proof}
\begin{lemma}
The map $\clMU$ is well defined modulo algebraic equivalence of cycles. 
\end{lemma}
\begin{proof}
%By  \cite{fulton} 10.3.2, 
The proof is exactly the same as in \cite{totaro}. We include it for completeness. Over an algebraically closed field, algebraic equivalence can be defined by connecting points via smooth projective curves. So let $C$ be a smooth projective curve $C$ over $k$ and let $j:W \subset X\times C$ be a subvariety of dimension $i+1$ such that the projection $f:W \to C$ is a dominant morphism. Let $p:X\times C \to X$ and $p_1:W \to X$ be the projections on the first factor. We have to show that for any two points $a,b \in C$
$$\clMU(W(a))= \clMU(W(b)) \in \hMU^{BM,\et}_{2i}(X)\otimes_{\hMU_{\ast}}\Zell$$ 
where, for $t \in C$, $W(t)$ denotes the image under $p$ of $f^{\ast}(t)$, the cycle associated to $f^{-1}(t) \subset X\times \{P\}$, as a closed subscheme of $X$. Since $\clMU$ commutes with pushforwards, it suffices to show 
\begin{equation}\label{tocheck}
\clMU(f^{\ast}(a))= \clMU(f^{\ast}(b)) \in \hMU^{BM,\et}_{2i}(W)\otimes_{\hMU_{\ast}}\Zell.
\end{equation}  
Let $\pi:W' \to \oW$ be an $\ell$-primary alteration of the projective compactification $\oW$ of $W$. It induces alterations $f^{-1}(a)'$ and $f^{-1}(b)'$ of degrees $d_a$ and $d_b$ respectively, such that $f^{-1}(0)'$ and $f^{-1}(b)'$ determine algebraically equivalent cycles of dimension $i$ in $W'$. Since the pushforwards under $\pi_{\ast}$ of the cycles $f^{\ast}(a)'$ and $f^{\ast}(b)'$ are just the cycles $f^{\ast}(a)$ and $f^{\ast}(b)$, it suffices to check equality (\ref{tocheck}) for $f^{\ast}(a)'$ and $f^{\ast}(b)'$. But  since algebraically equivalent cycles are also homologically equivalent, the images of $f^{-1}(a)'$ and $f^{-1}(b)'$ are equal in $H^2_{\et}(W';\Zell)$ and thus also in $\hMU^{BM,\et}_{2i}(W')\otimes_{\hMU_{\ast}}\Zell$ by Theorems \ref{quillenstheorem} and \ref{duality} for the smooth projective variety $W'$. This prove the lemma and finishes the proof of Theorem \ref{maintheorem}.
\end{proof}
%
%\begin{remark}
%If we had started with an arbitrary base field $k$, the same 
%
%\end{remark}
%
If $X$ is a smooth projective variety of dimension $n$, then cycles modulo rational and algebraic equivalence, graded by $Z^{\ast}X:=Z_{n-\ast}X$, form a ring. The same holds for the quotient of \'etale cobordism $\hMU_{\et}^{\ast}(X)\otimes_{\hMU^{\ast}} \Zell \cong \hMU^{\et}_{2n-\ast}(X)\otimes_{\hMU_{\ast}} \Zell$. The cycle map respects these products.
\begin{theorem}\label{lemma4.2}
The cycle map $\clMU$ is compatible with pullbacks along regular embeddings of smooth quasiprojective schemes, i.e. for any codimension $d$ regular embedding $X\to Y$ of smooth projective schemes, the following diagram of pullback maps commutes:
$$\xymatrix{
CH_iY \ar[d] \ar[r] & CH_{i-d}X \ar[d]\\
\hMU^{\et}_{2i}(Y)\otimes_{\hMU_{\ast}} \Zell \ar[r] & \hMU^{\et}_{2(i-d)}(X)\otimes_{\hMU_{\ast}} \Zell.}$$
\end{theorem}
\begin{cor}\label{products}
Let $X$ be a smooth projective variety over $k$. Then the map $\clMU$ commutes with products. 
\end{cor}
\begin{proof}
Recall the construction of the intersection product modulo rational equivalence by Fulton and MacPherson in \cite{fulton}. Let $\alpha$ and $\beta$ be two cycles on $X$. There is an external product cycle $\alpha \times \beta$ on $X \times X$ and the product  $\alpha \beta \in CH_{\ast}(X)$ is defined as the pullback of $\alpha \times \beta$ along the diagonal $CH_{\ast}(X\times X) \to CH_{\ast}(X)$. Similarly, there is an external product map $MU_{\ast}X\otimes_{MU_{\ast}} MU_{\ast}X \to MU_{\ast}(X\times X)$. The product in \'etale cobordism is defined by composing this map with the pullback along the diagonal. Since the diagonal is a regular embedding, the statement follows from the theorem above using the canonical isomorphism $\hEt(X\times X)\cong \hEt X \times \hEt X$ in $\hHh$.
\end{proof}

The rest of this section will be occupied with the proof of Theorem \ref{lemma4.2}. The argument is again the same as in \cite{totaro}. But since Totaro proves the statement in a slightly more general context using a Baum-Fulton-MacPherson pullback for possibly singular complex schemes, we include the proof in order to show that the main argument applies in our setting as well, where not all topological constructions are available.\\
The main example of a regular embedding is the inclusion of the zero-section of a vector bundle into the total space of the vector bundle $X \hookrightarrow E$. The pullback $CH_iE \to CH_iX$ is defined to be the inverse of the natural isomorphism $CH_iX \to CH_iE$ sending a subvariety $Z\subset X$ to $E|_Z \subset E$. For an arbitrary regular embedding $X \to Y$, the pullback map can be reduced to this example via the deformation to the normal cone. Namely, the pullback $CH_iY \to CH_iX$ sends by definition a subvariety $V \subset Y$ to the pullback $CH_iN_{X/Y} \to CH_iX$ we have just defined along the zero-section inclusion $X \into N_{X/Y}$ of the normal cone $C_{X\cap V}V$ of $V\cap X$ in $V$.\\
To prove Theorem {lemma4.2} let us start with the fundamental example of the zero-section embedding $X\into E$ of a vector bundle $E$ over a smooth quasiprojective scheme $X$. Since $E$ is embedded in its associated projective bundle, it is also quasiprojective. As in \cite{totaro}, we remark that since both pullback maps $CH_iE \to CH_{i-d}X$ and $\hMU^{BM,\et}_{2i}(E)\otimes_{\hMU_{\ast}} \Zell \to \hMU^{BM,\et}_{2(i-d)}(X) \otimes_{\hMU_{\ast}} \Zell$ are isomorphisms, it suffices to prove that the inverse maps commute:
\begin{equation}\label{inversesquare}
\xymatrix{
CH_{i-d}X \ar[d] \ar[r] & CH_iE \ar[d]\\
\hMU^{BM,\et}_{2(i-d)}(X)\otimes_{\hMU_{\ast}} \Zell \ar[r] & \hMU^{BM,\et}_{2i}(E)\otimes_{\hMU_{\ast}} \Zell.}\end{equation}
The map on the top row sends sends a subvariety $Z \subset X$ to the subvariety $E|_Z \subset E$. If $\pi:Z' \to Z$ is an $\ell$-primary alteration for $Z$, then $\pi^{\ast}E|_{Z'} \to E|_Z$ is an $\ell$-primary alteration too. It is now obvious from the definition of $\clMU$ that diagram (\ref{inversesquare}) commutes.\\
For a general embedding, we have seen that the cycle of a subvariety $V \subset Y$ in $CH_iV$ is sent to the pullback of $C_{V\cap X}V \subset N_{X/Y}$ in $CH_{i-d}X$. The inclusion of $X$ in $N_{X/Y}$ has just been checked. Hence to prove the theorem, it remains to show that the pullback of $\clMU(V)\in \hMU^{\et}_{2i}(Y)\otimes_{\hMU^{\ast}} \Zell$ to $X$ is equal to the pullback of $\clMU(C_{V\cap X}V)\in \hMU^{\et}_{2i}(N_{X/Y})\otimes_{\hMU^{\ast}} \Zell$ to $X$.\\
This follows from the deformation to the normal cone of regular embeddings, see \cite{fulton} chapter 5. For a closed subscheme $X\subset Y$, there is a scheme $M_XY$ together with a closed embedding of $X\times \Pro^1$ in $M_XY$ and a flat morphism $\rho: M_XY \to \Pro^1$ such that
$$\xymatrix{
X\times \Pro^1 \ar[dr] \ar[r] & M_XY \ar[d]\\
 & \Pro^1}$$
commutes. Moreover, for $t\in \Pro^1 -\{ \infty \}=\A^1$, we have $\rho^{-1}(t)\cong Y$ and the embedding $X\subset \rho^{-1}(t)$ is the given embedding $X\subset Y$, and over $\infty \in \Pro^1$, the embedding $X\subset \rho^{-1}(\infty)$ is the zero-section embedding of $X$ into the normal cone $C_XY$ of $X$ in $Y$. Since $X$ is smooth over $k$, $\rho$ is a smooth morphism and $M_XY$ is smooth. It is also quasiprojective over $k$ since it is an open subscheme of the blowup of $Y\times \Pro^1$ along $X\times 0$.\\ 
If $V$ is a subvariety of $Y$ of dimension $i$, we construct in the same way a scheme $M_{X\cap V}V$ which is a subvariety of $M_XY$. We want to show that $\clMU(V)\in \hMU^{BM,\et}_{2i}(Y)\otimes_{\hMU_{\ast}}\Zell$ is the pullback of $\clMU(M_{X\cap V}V) \in \hMU^{BM,\et}_{2(i+1)}(M_XY)\otimes_{\hMU_{\ast}}\Zell$ along the regular embedding $Y \into M_XY$ and that $\clMU(C_{X\cap V}V)$ is also the pullback of $\clMU(M_{X\cap V}V)$ along the regular embedding $N_{X/Y}\into M_XY$. The commutative diagram
$$\xymatrix{
Y \ar[r] & M_XY & \ar[l] N_{X/Y}\\
X \ar[u] \ar[r] & X\times \Pro^1 \ar[u] &\ar[l] X \ar[u]}$$
then implies that the pullback of $\clMU(V)$ and of $\clMU(C_{X\cap V}V)$ to $X$ are both the pullback of the same element $\clMU(M_{X\cap V}V)$ along the map $X \to X\times \Pro^1 \to M_XY$.\\
Hence it suffices to prove the following lemma.
\begin{lemma}\label{lemma4.3}
Let $W$ be an $n$-dimensional quasiprojective smooth variety over $k$, $T$ a smooth curve, $f:W\to T$ a nonconstant flat morphism, $t\in T$. Let $Z$ be a closed subvariety of $W$ of dimension $i$ and let $g=f_{|Z}:Z \to T$ be the restriction of $f$ to $Z$. Then the pullback of the class of $Z$ in $\hMU^{BM,\et}_{2i}(W)\otimes_{\hMU_{\ast}} \Zell$ equals the class of the cycle associated to the scheme $g^{-1}(t)$ in $\hMU^{BM,\et}_{2i-2}(f^{-1}(t))\otimes_{\hMU_{\ast}} \Zell$.
\end{lemma}
\begin{proof}
Since $f$ is flat and the embedding of $\{t\} \into T$ is regular, the inclusion of the subscheme $e_t:f^{-1}(t) \subset W$ is a codimension-one regular embedding of smooth schemes. By Proposition \ref{pullback}, the pullback $e_t^{\ast}: \hMU^{BM,\et}_{\ast}(W) \to \hMU^{BM,\et}_{\ast-2}(f^{-1}(t))$ is defined by cap product with the orientation class $\omega_{f^{-1}(t),W}\in \hMU^2_{\et}(W,W-f^{-1}(t))$ induced by the Thom class of the normal bundle and homotopy purity. By naturality of Thom classes, we have 
$\omega_{f^{-1}(t),W} =f^{\ast}\omega_{t,T}$, where $\omega_{t,T} \in \hMU^2_{\et}(T,T-t)$ is the orientation class of the embedding $\{t\} \into T$. Hence we have to show:
$$f^{\ast}\omega_{t,T} \cap \clMU(Z)=\clMU(g^{-1}(t))~\mathrm{in}~\hMU^{BM,\et}_{2i-2}(f^{-1}(t)) \otimes_{\hMU_{\ast}} \Zell.$$
Since the class of $g^{-1}(t)$ is defined as the pushforward of the associated fundamental class along $g^{-1}(t) \into f^{-1}(t)$, it suffices to prove the formula  
$$g^{\ast}\omega_{t,T} \cap \clMU(Z)=\clMU(g^{-1}(t))~\mathrm{in}~\hMU^{BM,\et}_{2i-2}(g^{-1}(t)) \otimes_{\hMU_{\ast}} \Zell.$$
So let $\pi:Z' \to Z$ be a smooth alteration of $Z$ of degree $d$ prime to $\ell$. By \cite{fulton}, p. 34, we have the identity of cycles on the scheme $g^{-1}(t)$
$$\pi_{\ast}((g\pi)^{-1})(t)=g^{-1}(t).$$
Using the projection formula and the definition of $\clMU(Z)$, we obtain the following identification in $\hMU^{BM,\et}_{2i-2}(g^{-1}(t)) \otimes_{\hMU_{\ast}} \Zell$
$$\begin{array}{rcl}
\frac{1}{d}\pi_{\ast}((g\pi)^{\ast}\omega_{t,T} \cap [Z']) &= & \frac{1}{d}\pi_{\ast}(\pi^{\ast}g^{\ast}\omega_{\{t\},T} \cap [Z'])\\
 & = & g^{\ast}\omega_{\{t\},T} \cap \frac{1}{d}\pi_{\ast}[Z']\\
 & = & g^{\ast}\omega_{\{t\},T} \cap \clMU(Z).
\end{array}$$
Hence it suffices to show
$$(g\pi)^{\ast}\omega_{t,T} \cap [Z']=[(g\pi)^{-1}(t)]~\mathrm{in}~\hMU^{BM,\et}_{2i-2}((g\pi)^{-1}(t)) \otimes_{\hMU_{\ast}} \Zell.$$
So replacing $g$ by $g\pi$, we may assume $Z$ is projective and smooth of dimension $i$ over $k$.\\
For such a $Z$, we know that $\hMU^{BM,\et}_{2i-2}(g^{-1}(t)) \otimes_{\hMU_{\ast}} \Zell$ is isomorphic to the Borel-Moore homology $H^{BM,\et}_{2i-2}(g^{-1}(t), \Zell)$ by Theorems \ref{quillenstheorem} and \ref{duality}. Hence it suffices to prove the above formula in \'etale homology $H^{BM,\et}_{2i-2}(g^{-1}(t), \Zell)$ which has been done by Friedlander in \cite{fried}, Proposition 17.4. This proves the lemma and Theorem \ref{lemma4.2}.
\end{proof}
\section{Examples via Godeaux-Serre varieties} 
Let $k$ be an algebraically closed field and $\ell$ a prime different from the characteristic of $k$. Let $G$ be a finite $\ell$-group. For any given integer $r\geq 1$, Serre \cite{serre} has shown the existence of a representation $V$ of rank $n+1$ over $k$ and a smooth variety $Y \subset \Pro(V)$ over $k$ such that:\\
(a) $G$ acts without fixed points on $Y$;\\
(b) $Y$ is a complete intersection of a number of hypersurfaces of $\Pro(V)$ of degree $d$ which are smooth on $Y$ and intersect transversally;\\
(c) $\dim_k Y=r$;\\
(d) $X:=Y/G$ is a smooth projective variety over $k$;\\
(e) we observe that (b) and the weak Lefschetz theorem imply that the $\ell$-adic \'etale cohomology of $Y$ is isomorphic to the one of $\Pro(V)$ up to dimension $r-1$.\\  
We will apply Serre's construction in the following two cases.
\subsection{The examples of Atiyah and Hirzebruch}
Let us first review the argument of Atiyah and Hirzebruch in the light of \'etale homotopy theory. We can formulate the following analogue of Proposition 6.6 of \cite{atiyahhirzebruch}.
\begin{prop}\label{prop6.6}
Let $k$ and $G$ be as above. For any positive integer $n>2$, there is a smooth projective variety $X$ such that the continuous $\Zell$-cohomology of $\hEt X$ is equal to the cohomology of the product of Eilenberg-MacLane spaces $K(\Zell,2)\times BG$ in $\hSh$ up to dimension $n$. In particular, $H^{\ast}(G;\Zell)$ is a direct factor of $H_{\et}^{\ast}(X;\Zell)$ up to dimension $n$. 
\end{prop}
\begin{proof}
The proof is the analogue of the one in \cite{atiyahhirzebruch}. The remarkable thing is that it can be reformulated in terms of \'etale homotopy. We choose $r-1\geq n$ and define $Y\subset \Pro(V)$ and $X=Y/G$ as above. Let $\Oh(1) \to \Pro(V)$ be the canonical quotient line bundle on $\Pro(V)$ and let $\eta \to Y$ be its pullback to $Y$. Since $G$ acts on these bundles, there is a bundle $\xi$ on $X$ such that $\eta=\pi^{\ast}\xi$, where $\pi:Y \to X$ is the covering map. %It is a principal $G$-fibration of schemes as in cite{fried}. 
We denote by $u\in H^2(\hEt X;\Zell)$ the first Chern class of $\xi$ and set $v:=\pi^{\ast}u \in H^2(\hEt Y;\Zell)$. Let $f: \hEt X\to K(\Zell,2)$ be a map in $\hHh$ representing $u$, $g:\hEt X\to BG$ be the map induced by the principal $G$-fibration $\hEt Y \to \hEt X$ and let $\bar{g}: \hEt Y\to EG$ be the covering map of $g$ in $\hHh$, see \cite{ensprofin} and \cite{profinhom}. Then $(f\circ \pi,\bar{g}):\hEt Y \to K(\Zell,2)\times EG$ is the covering map of $(f,g):\hEt X\to K(\Zell,2)\times BG$ in $\hHh$. Since the cohomology of $EG$ is trivial and since the cohomology of $\hEt \Pro(V)$ is isomorphic to the one of $K(\Zell,2)$ up to dimension $r-1$, (e) above implies that $(f\circ \pi,\bar{g})$ induces an isomorphism in $\Zell$-cohomology up to dimension $r-1$. Hence $(f,g):\hEt X\to K(\Zell,2)\times BG$ induces an isomorphism in continuous $\Zell$-cohomology up to dimension $r-1$. 
\end{proof}
Now the same proof as in \cite{atiyahhirzebruch}, Proposition 6.7, or in \cite{ctsz}, Th\'eor\`eme 2.1, shows the following proposition. 
\begin{prop}\label{prop6.7}
Let $G$ be the group $(\Zl)^3$. There is a cohomology class $y\in H^{4}(G;\Zell)$ of order $\ell$ and a cohomology operation of odd degree that does not vanish on $y$.
\end{prop}
For $\ell=2$, the Steenrod operation of Proposition \ref{prop6.7} is the Steenrod square $Sq^3$ of degree $3$. For $\ell$ odd, the operation is $\beta\Ph^1$ of degree $2\ell-1$, where we denote by $\Ph^1:H^i(G;\Zl) \to H^{i+2(\ell-1)}(G;\Zl)$ the first $\ell$th power operation and by $\beta:H^i(G;\Zl)\to H^{i+1}(G;\Zl)$ the Bockstein operator induced by the short exact sequence 
$$0\to \Zell \to \Zell \to \Zl \to 0.$$
Since all odd degree cohomology operations vanish on the image image of the map $\theta:\hMU_{\et}^{\ast}(X)\otimes_{\hMU^{\ast}} \Zell \to H_{\et}^{\ast}(X;\Zell)$, Propositions \ref{prop6.6} and \ref{prop6.7} show that there is a cohomology class in $H_{\et}^{4}(Y/G;\Zell)$ that is not algebraic, which proves part (b) of Corollary \ref{obstruction}.\\
Steenrod operations on $H_{\et}^{\ast}(X;\Zl)$ have been constructed by Raynaud in \cite{raynaud}. They can be also constructed using \'etale homotopy theory. For any finite groups $\pi$ and $G$ and any positive integers $n$ and $m$, there is a bijection between the set of cohomology operations $H^n(-;\pi) \to H^m(-;G)$ of continuous chomology of profinite spaces and the set of maps $\Hom_{\hHh}(K(\pi,n),K(G,m))$ in $\hHh$. But since $\pi$ and $G$ are finite groups, their Eilenberg MacLane spaces are simplicial finite sets. So we have 
$$\Hom_{\hHh}(K(\pi,n),K(G,m))=\Hom_{\Hh}(K(\pi,n),K(G,m)).$$
Applied to $\hEt X$ of a scheme, this defines all primary cohomology operations on the \'etale cohomology of $X$. But in $\hSh$ also higher cohomology operations can be constructed just as in $\Sh$. This shows that also higher operations exist for \'etale cohomology via the \'etale realization functor.\\ 
Finally, we remark that Soul\'e and Voisin showed in \cite{soulevoisin}, Theorem 1, that Totaro's topological obstruction can only detect non-algebraic torsion cohomology classes whose order is small relative to the dimension of the variety. Their proof applies in the same way in our situation in positive characteristic. Hence with this method we cannot expect to find non-algebraic classes with an arbitrary order compared to the dimension.
\subsection{Nontrivial elements in the Griffiths group}
Totaro has analyzed the kernel of the map $MU^{\ast} X\otimes_{MU^{\ast}} \Z/2 \to H^{\ast}(X,\Z/2)$ over $\C$ and constructed elements in the kernel that are in the image of the cycle map \cite{totaro}. By transferring the argument to \'etale homotopy theory, we now check that these varieties  provide examples over any algebraically closed field of characteristic $\neq 2$ such that the classical cycle map to \'etale cohomology is not injective. They correspond to the examples of \cite{totaro}, Theorems 7.1 and 7.2. We remind the reader of the comment in the introduction for other examples of the non-injectivity of the map in (a) below.
\begin{prop}
Let $k$ be an algebraically closed field of characteristic $\neq 2$.\\ 
{\rm (a)} There is a smooth projective variety over $k$ of dimension $7$ such that the map $CH^2X/2 \to H_{\et}^4(X,\Z/2)$ is not injective.\\
{\rm (b)} There is a smooth projective variety over $k$ of dimension $15$ and an element $\alpha \in CH^3X$ such that $2\alpha=0$ and $\alpha$ is homologically but not algebraically equivalent to zero.  
\end{prop} 
\begin{proof}
We equip $\hSh$ with the $\Z/2$-model structure. For both varieties we apply Serre's construction to a suitable $2$-group $G$ to get a variety $Y$ of dimension $r$ with a $G$-action as above. By Proposition \ref{prop6.6}, $\hEt Y/G \to BG\times K(\Zell,2)$ induces an isomorphism in $\Zl$-cohomology up to dimension $r-1$. Hence $\hEt Y/G$ contains the $r-1$-skeleton $\sk_{r-1}BG$ of $BG$ up to weak equivalence. Thus it suffices to find elements in $\hMU^{\ast}(\sk_{r-1}BG)\otimes_{\hMU^{\ast}}\Z/2$. Note that since $\sk_{r-1}BG$ is a finite simplicial set, the proof of Theorem \ref{quillenstheorem} implies 
\begin{equation}\label{BGisom}
\hMU^{\ast}(\sk_{r-1}BG)\otimes_{\hMU^{\ast}}\Z/2\cong MU^{\ast}(\sk_{r-1}BG)\otimes_{MU^{\ast}}\Z/2.\end{equation}
For (a), we take $r=8$ and let $G$ be the extra-special $2$-group $D(2)$ of order $32$ with center $\Z/2$ of section 5 of \cite{totaro}. The dimension of the associated Serre variety $X=Y/G$ is then $7$. Using identification (\ref{BGisom}), the same argument as in the proof of Theorem 7.1 of \cite{totaro} now applies. We only have to observe that the representations $A$ and $B$ of $G$, which come from representations of $SO(4,k)$ under the restriction $G\subset SO(4,k)$, can be defined over $k$. They yield $k$-vector bundles of ranks $3$ and $4$, respectively. Since Chow groups define an oriented cohomology theory in the sense of \cite{lm}, there are Chern classes of vector bundles in $CH^{\ast}X$. After taking second Chern classes, we get a nonzero cycle $c_2A-c_2B$ in $CH^2X$ over $k$. Since \'etale cobordism is also an oriented cohomology theory by Theorem \ref{orientation}, there are Chern classes of vector bundles as well. By \cite{totaro}, the image of the cycle $c_2A-c_2B$ is nonzero in $\hMU^4X\otimes_{\hMU^{\ast}} \Z/2$, but is mapped to zero in $H_{\et}^4(X,\Z/2)$. \\
For (b), we chose $r=16$ and $G$ be $D(2)\times \Z/2$, where $D(2)$ is as in (a). Again, using (\ref{BGisom}), the arguments of \cite{totaro}, Theorem 7.2, apply to get a cycle on $X$ over $k$.
\end{proof}

This finishes the proof of Corollary \ref{obstruction}.
\bibliographystyle{amsalpha}

Mathematisches Institut, Universit\"at M\"unster, Einsteinstr. 62, D-48149 M\"unster\\
E-mail address: gquick@math.uni-muenster.de\\
Homepage: www.math.uni-muenster.de/u/gquick

\end{document}